\documentclass[11pt, a4paper]{amsart}
\usepackage{amsfonts, amssymb, amsmath, amscd, amsthm}
\usepackage{color, hyperref, graphicx}
\usepackage[all]{xy}
\usepackage[english]{babel}
\usepackage[utf8]{inputenc}
\usepackage[T1]{fontenc}
\usepackage{lmodern}
\usepackage{setspace}
 
\usepackage{etoolbox}
\patchcmd{\section}{\scshape}{\bfseries}{}{} 
\makeatletter
\renewcommand{\@secnumfont}{\bfseries}
\makeatother

\DeclareFontFamily{U}{BOONDOX-calo}{\skewchar\font=45}
\DeclareFontShape{U}{BOONDOX-calo}{m}{n}{
  <-> s*[1.05] BOONDOX-r-calo}{}
\DeclareFontShape{U}{BOONDOX-calo}{b}{n}{
  <-> s*[1.05] BOONDOX-b-calo}{}
\DeclareMathAlphabet{\mathcalboondox}{U}{BOONDOX-calo}{m}{n}
\SetMathAlphabet{\mathcalboondox}{bold}{U}{BOONDOX-calo}{b}{n}
\DeclareMathAlphabet{\mathbcalboondox}{U}{BOONDOX-calo}{b}{n}

\newcommand{\suchthat}{\;\ifnum\currentgrouptype=16 \middle\fi|\;}
\numberwithin{equation}{section}
\allowdisplaybreaks

\theoremstyle{plain}
	\newtheorem{lemma}{Lemma}
	\newtheorem*{lemma*}{Lemma}
	\newtheorem{theorem}[lemma]{Theorem}
	\newtheorem*{theorem*}{Theorem}
	\newtheorem{corollary}[lemma]{Corollary}
	\newtheorem*{corollary*}{Corollary}
	
	\newtheorem*{proposition*}{Proposition}

\title[An extension result for maps admitting an AAT]{An extension result for maps admitting an algebraic addition theorem}

\date{\today}

\author{E. Baro}
\address{Departamento de \'Algebra, Facultad de Ciencias Matem\'aticas, Universidad Complutense de Ma\-drid, 28040 Madrid (Spain)}
\curraddr{}
\email{eliasbaro@pdi.ucm.es}
\author{J. de Vicente \and M. Otero}
\address{Departamento de Matem\'aticas, Universidad Aut\'onoma de Madrid, 28049 Madrid (Spain)}
\curraddr{}
\email{juan.devicente@uam.es, margarita.otero@uam.es}


\thanks{All the authors supported by Spanish GAAR MTM2011-22435 and MTM2014-55565. 
Second author also supported by a grant of the International Program of Excellence in Mathematics at 
Universidad Aut\'onoma de Madrid.}

\begin{document}

\addtolength{\textheight}{-3.5mm}
\addtolength{\footskip}{3.5mm}
\pagestyle{plain}

\begin{abstract}
We prove that if an analytic map $f:=(f_1,\ldots ,f_n):U\subset \mathbb{C}^n\rightarrow \mathbb{C}^n$ admits an 
algebraic addition theorem then there exists a meromorphic map $g:=(g_1,\ldots ,g_n):\mathbb{C}^n\dasharrow \mathbb{C}^n$ admitting an 
algebraic addition theorem such that $f_1,\ldots ,f_n$ are algebraic over $\mathbb{C}(g_1,\ldots ,g_n)$ on $U$ 
(this was proved by K. Weierstrass in dimension $1$).
Furthermore, $(g_1,\ldots ,g_n)$ admits a \emph{rational} addition theorem.
\end{abstract}

\maketitle

\section{Introduction}\label{sec:intro}

The aim of this paper is to study maps admitting an algebraic addition theorem, maps whose coordinate functions can be viewed as 
limitting (degenerate) cases of abelian functions.
Let $\mathbb{K}$ be $\mathbb{C}$ or $\mathbb{R}$ and $\mathcalboondox{M}_{\mathbb{K},n}$ be the quotient field of 
$\mathcal{O}_{\mathbb{K},n}$, the ring of power series in $n$ variables with coefficients in $\mathbb{K}$ that 
are convergent in a neighborhood of the origin.

\medskip

\noindent {\bf Definition.} Let $u$ and $v$ be variables of $\mathbb{C}^n$.
We say $(\phi _1,\ldots ,\phi _n)\in \mathcalboondox{M}_{\mathbb{K},n}^n$ admits an \emph{algebraic addition theorem} (AAT) if 
$\phi _1,\ldots ,\phi _n$ are algebraically independent over $\mathbb{K}$ and if each $\phi _i(u+v)$, $i=1,\ldots ,n$, is 
algebraic over 
\[
\mathbb{K}(\phi _1(u),\ldots ,\phi _n(u),\phi _1(v),\ldots ,\phi _n(v)).
\]

\medskip

The concept of AAT was introduced by K. Weierstrass during his lectures on abelian functions in Berlin in $1870$ (see \cite{Weierstrass}).
He stated that the coordinate functions of a \emph{global} meromorphic map admitting an AAT are either abelian functions or 
degenerate abelian functions.
He proved it for dimension $1$ and F. Severi in \cite{Severi} (see also Y. Abe \cite{Abe,Abe_errata}) for dimension $n$.
Weierstrass also proved the following \emph{extension result}: the germ of an analytic function admitting an AAT can be transformed 
algebraically into the germ of a global function admitting an AAT; and he stated, without a proof, an $n$-dimensional version of it.
As far as we know, no such proof existed in the literature so far. 
We prove it here as a consequence (Corollary \ref{C2}) of the main result of the paper, that we now state.

\begin{theorem}[Extension Theorem]\label{T2}
Let $\phi :=(\phi _1,\ldots ,\phi _n)\in \mathcalboondox{M}_{\mathbb{K},n}^n$ admit an AAT.
Then, there exist $\psi :=(\psi _1,\ldots ,\psi _n)\in \mathcalboondox{M}_{\mathbb{K},n}^n$ admitting an AAT and algebraic over 
$\mathbb{K}(\phi)$, and an additional meromorphic series $\psi _0\in \mathcalboondox{M}_{\mathbb{K},n}$ algebraic over 
$\mathbb{K}(\psi )$ such that,
\begin{enumerate}
\item[$(1)$] For each $f(u)\in \mathbb{K}\big(\psi _0(u),\ldots ,\psi _n(u)\big)$,
\smallskip
\begin{enumerate}
\item[$(a)$] $f(u+v)\in \mathbb{K}\big(\psi _0(u),\ldots ,\psi _n(u),\psi _0(v),\ldots ,\psi _n(v)\big)$ and
\smallskip
\item[$(b)$] $f(-u)\in \mathbb{K}\big(\psi _0(u),\ldots ,\psi _n(u)\big)$.
\end{enumerate}
\smallskip
\item[$(2)$] Each $\psi _0,\ldots ,\psi _n$ is the quotient of two convergent power series whose complex domain of convergence is 
$\mathbb{C}^n$.
\end{enumerate}
\end{theorem}

\begin{corollary}\label{C2}
Any $\phi \in \mathcalboondox{M}_{\mathbb{K},n}^n$ admitting an AAT is algebraic over $\mathbb{K}(\psi)$ for some 
$\psi \in \mathcalboondox{M}_{\mathbb{K},n}^n$ admitting an AAT and whose coordinate functions are the quotient of two convergent power 
series whose complex domain of convergence is $\mathbb{C}^n$. 
\end{corollary}

We point out that this theorem gives not only an extension result, but also a uniform rational version of the AAT. 
In fact, given a $\phi$ admitting an AAT, we obtain the rational version in Theorem \ref{T2}\,(1a) through the coefficients of the 
polynomial asssociated to each $\phi _i(u+v)$.
Then, we obtain the extension result of Theorem \ref{T2}\,(2) by considering the rational expression obtained in Theorem \ref{T2}\,(1a).
In particular, this shows that any $\phi$ admitting an AAT can be analytically extended to a multivalued analytic map with a finite number of 
branches.
Thus, we provide a new way of proving Weierstrass' extension result in dimension $1$, whose classical proofs go the other way around
(and do not provide a rational counterpart):
first it is proved the finiteness of the number of branches of the extension of such analytic $\phi$ and then, making use of the coefficients 
of the relevant polynomials, it is given a global univaluated meromorphic function admitting an AAT.

The motivation of the results of this paper is to study abelian locally $\mathbb{K}$-Nash groups, for $\mathbb{K}=\mathbb{R}$ or 
$\mathbb{C}$.
Charts at the identity of such groups admit an AAT. 
Locally Nash groups (i.e. for $\mathbb{K}=\mathbb{R}$) were studied by J.J. Madden and C.M. Stanton \cite{Madden_Stanton} and M. Shiota
\cite{Shiota2}, mainly in dimension $1$. 
In particular, the Extension Theorem will allow us to reduce the study of simply connected abelian locally Nash groups to those whose 
charts are restrictions of (global) meromorphic functions admitting an AAT (see \cite{Baro_deVicente_Otero}).

The results of this paper are part of the second author's Ph.D. dissertation.

\section{The Extension Theorem.}\label{SPreliminaries}\label{SAlgebra}

For each $\epsilon >0$, let $U_{\mathbb{K},n}(\epsilon):=\{ a\in \mathbb{K}^n \suchthat \| a\| <\epsilon \}$.
\emph{We will only consider convergence over open subsets of $\mathbb{C}^n$, let $U_n(\epsilon ):=U_{\mathbb{C},n}(\epsilon)$.}
We say that $(\phi _1,\ldots ,\phi _m) \in \mathcalboondox{M}_{\mathbb{K},n}^m$ is {\itshape convergent} in $U_n(\epsilon )$ if each 
$\phi _1,\ldots ,\phi _m$ is the quotient of two power series convergent on $U_n(\epsilon )$.

As usual, by the identity principle for analytic functions, we identify $\mathcal{O}_{\mathbb{K},n}$ with the ring of germs of analytic 
functions at $0$, and $\mathcalboondox{M}_{\mathbb{K},n}$ with its quotient field.
We will use without mention properties of $\mathcal{O}_{\mathbb{K},n}$, see e.g. R.C. Gunning and H. Rossi \cite{Gunning_Rossi} and 
J.M. Ruiz \cite{Ruiz}.

Let $\epsilon >0$.
Let $\phi :=(\phi _1,\ldots ,\phi _m)\in \mathcalboondox{M}_{\mathbb{K},n}^m$ be convergent on $U_n(\epsilon )$, let 
$a\in U_{\mathbb{K},n}(\epsilon )$ and let $(u,v):=(u_1,\ldots ,u_n,v_1,\ldots ,v_n)$ be a $2n$-tuple of variables. 
We will use the following notation:

\[
\begin{array}{l}
\phi _{(u,v)}:=\big(\phi _1(u),\ldots ,\phi _m(u), \phi _1(v),\ldots ,\phi _m(v)\big)
\in \mathcalboondox{M}_{\mathbb{K},2n}^{2m}.\\
\phi _{u+v}:=\big(\phi _1(u+v),\ldots ,\phi _m(u+v)\big)
\in \mathcalboondox{M}_{\mathbb{K},2n}^m.\\
\phi _{u+a}:=\big(\phi _1(u+a),\ldots ,\phi _m(u+a)\big)
\in \mathcalboondox{M}_{\mathbb{K},n}^m.
\end{array}
\]\label{AAT_equivalence}

Given $\phi \in \mathcalboondox{M}_{\mathbb{K},p}^{n}$ and $\psi \in \mathcalboondox{M}_{\mathbb{K},p}^{m}$ we say that the tuple 
$\phi$ is {\itshape algebraic} over $\mathbb{K}(\psi):=\mathbb{K}(\psi _1 , \ldots ,\psi _m)$ if each component, 
$\phi _1,\ldots ,\phi _n$, is algebraic over $\mathbb{K}(\psi)$.

\begin{center}
\emph{Thus, $\phi \in \mathcalboondox{M}_{\mathbb{K},n}^n$ admits an \emph{algebraic addition theorem (AAT)} if 
$\phi _1,\ldots ,\phi _n$ are algebraically independent over $\mathbb{K}$ and $\phi _{u+v}$ is algebraic over 
$\mathbb{K}(\phi _{(u,v)})$.}
\end{center}

Note that if $\phi \in \mathcalboondox{M}_{\mathbb{R},n}$ admits an AAT then $\phi$ also admits an AAT when considered as an element 
of $\mathcalboondox{M}_{\mathbb{C},n}$.\label{AAT_equivalence2}

We first prove two properties of maps admitting an AAT.

\begin{lemma}\label{algebraic} 
Let $\epsilon >0$ and let $\phi \in \mathcalboondox{M}_{\mathbb{K},n}^n$ be convergent on $U_n(\epsilon)$.
If $\phi$ admits an AAT then $\phi _{u+a}$ is algebraic over $\mathbb{K}(\phi)$, for each 
$a\in U_{\mathbb{K},n}(\epsilon)$.
\end{lemma}
\begin{proof}
Fix $j\in \{1,\ldots ,n\}$ and let $f(u,v):=\phi _j(u+v)$.
By hypothesis, there exists $P\in \mathbb{K}[X_1,\ldots ,X_{2n}][Y]$ such that 
$P(\phi (u),\phi (v);Y)\neq 0$ and $P(\phi (u),\phi (v);f(u,v))=0$.
For any $a\in U_{\mathbb{K},n}(\epsilon)$ such that $P(\phi (u),\phi (a),Y)$ is not identically zero, 
we clearly obtain that $f(u,a)$ is algebraic over $\mathbb{K}(\phi)$.
We have to consider those $a\in U_{\mathbb{K},n}(\epsilon)$ such that $P(\phi (u),\phi (a);Y)$ is identically zero.

We first check that there exists an open dense subset $U$ of $U_{\mathbb{K},n}(\epsilon )$ such that for each 
$a\in U$, $P(X_1,\ldots ,X_n,\phi (a);Y)\in \mathbb{K}[X_1,\ldots ,X_n][Y]$ is a non-zero polynomial.
Let $W$ be an open dense subset of $U_{\mathbb{K},n}(\epsilon )$ such that
\[
W\subset \{  a\in U_{\mathbb{K},n}(\epsilon )\suchthat \phi (a)\in \mathbb{K}^n \}
\]
and $\phi :W\rightarrow \mathbb{K}^n$ is analytic.
Let 
\[
U:=\{ a\in W \suchthat P(X_1,\ldots ,X_n,\phi (a);Y) \neq 0\}. 
\]
Since $W$ is an open dense subset of $U_{\mathbb{K},n}(\epsilon )$, it is enough to show that $W\setminus U$ is closed 
and nowhere dense in $W$.
Clearly $W\setminus U$ is closed in $W$ because $\phi$ is continuous in $W$.
To prove the density, we note that if $W\setminus U$ contains an open subset of $W$ then
\[
\{ a\in U_{\mathbb{K},n}(\epsilon )\suchthat P(\phi (u),\phi (a);Y)\in \mathcalboondox{M}_{\mathbb{K},n+1} \text{ and } 
P(\phi (u),\phi (a);Y)=0\}
\]
contains an open subset of $U_{\mathbb{K},n}(\epsilon )$ and therefore $P(\phi (u),\phi (v);Y)=0$, a contradiction.

To finish the proof we will show that for each $a\in U_{\mathbb{K},n}(\epsilon )$, there exists 
$Q_{a}\in \mathbb{K}[X_1,\ldots ,X_n][Y]$ such that $Q_{a} (\phi (u);Y)$ is not identically zero and  
$Q_{a} (\phi (u);f(u,a))=0$.
We follow the proof of \cite[Ch. IX. \textsection 5. Theorem 5]{Bochner_Martin}.
For each $a\in U$, where $U$ is as above, let
\[
P_{a}(X_1,\ldots ,X_n;Y)=\sum _{i,\mu \leq N} b_{i,\mu ,a} \ X_1^{\mu _1}\ldots X_n^{\mu _n}Y ^i
\]
denote the polynomial $P(X_1,\ldots ,X_n,\phi (a);Y)$.
We have that $U$ is dense in $U_{\mathbb{K},n}(\epsilon )$ and $P_{a}\neq 0$ for all $a\in U$.
For each $a\in U$, we define
\[
E(P_{a}):= \sum _{i,\mu \leq N} \| b_{i,\mu ,a} \| ^2.
\]
We note that $E(P_{a})>0$, for all $a\in U$.
For each $a\in U$, let 
\[
Q_{a}(X_1,\ldots ,X_n;Y):=\sum _{i,\mu \leq N} c_{i,\mu ,a} \ X_1^{\mu _1}\ldots X_n^{\mu _n}Y ^i,
\]
where
\[
c_{i,\mu ,a}:=\frac{b_{i,\mu ,a}}{\sqrt{E(P_{a})}}. 
\]
Hence, for each $a\in U$, we have that $Q_{a}(\phi (u);Y)$ is not identically zero, $Q_{a}(\phi (u);f(u,a))=0$ and $E(Q_{a})=1$.
We define
\[
\vec{v}(a):=(c_{i,\mu ,a})_{i,\mu \leq N}\in \{ z\in \mathbb{K}^{(N+1)^{(n+1)}}\suchthat \|z\|=1 \}.
\]
Take $a\in U_{\mathbb{K},n}(\epsilon)\setminus U$.
Since $U$ is an open dense subset of $U_{\mathbb{K},n}(\epsilon )$, there exists a sequence $\{a_k\}_{k \in \mathbb{N}}\subset U$ 
that converges to $a$.
For each $a_k$, the identity $Q_{a_k}(\phi (u) ;f(u,a_k))=0$ holds, therefore
\[
\sum _{i,\mu \leq N} c_{i,\mu ,a_k}\phi _1( u) ^{\mu _1}\ldots \phi _n(u)^{\mu _n} f (u,a_k)^i=0.
\]
By hypothesis there are $\alpha ,\beta \in \mathcal{O}_{\mathbb{K},2n}$, $\beta \neq 0$, convergent on 
$U_{2n}(\epsilon)$, such that $f(u,v) =\frac{\alpha (u,v)}{\beta (u,v)}$ and $\beta (u,a)\neq 0$ for 
all $a\in U_{\mathbb{K},n}(\epsilon )$.
In particular
\begin{equation}
\sum _{i,\mu \leq N} c_{i,\mu ,a_k}\phi _1(u) ^{\mu _1}\ldots \phi _n(u)^{\mu _n} 
\alpha (u,a_k)^i\beta (u,a_k)^{N-i}=0.\label{main}
\end{equation}
Since $\{ z\in \mathbb{K}^{(N+1)^{(n+1)}}\suchthat \| z\|=1 \}$ is compact, taking a suitable subsequence we can assume that the 
sequence $\{\vec{v}(a_k)\}_{k\in \mathbb{N}}$ is convergent.
For each $i,\mu \leq N$, we define
\[
c_{i,\mu ,a}:=\lim _{k \rightarrow \infty} \ c_{i,\mu ,a_k}.
\]
Since $\alpha$ and $\beta$ are continuous, when $k$ tends to infinity equation (\ref{main}) becomes
\[
\sum _{i,\mu \leq N} c_{i,\mu ,a}\phi _1(u) ^{\mu _1}\ldots \phi _n(u)^{\mu _n} 
\alpha (u,a)^i\beta (u,a)^{N-i}=0.
\]
So dividing by $\beta (u,a)^N$, we also have
\[
\sum _{i,\mu \leq N} c_{i,\mu ,a}\phi _1(u) ^{\mu _1}\ldots \phi _n(u)^{\mu _n} f(u,a)^i=0
\]
and hence the polynomial
\[
Q_a(X_1,\ldots ,X_n;Y):= \sum _{i,\mu \leq N} c_{i,\mu ,a}X_1 ^{\mu _1}\ldots X_n^{\mu _n} Y^i
\]
satisfies $Q_a(\phi (u) ,f(u,a))=0$.
We note that $E(Q_a)=\lim _{k\rightarrow \infty} E(Q_{a_k}) =1$, so $Q_a\neq 0$.
Since $\phi _1,\ldots ,\phi _n$ are algebraically independent over $\mathbb{K}$ and $Q_a(X_1,\ldots ,X_n,Y)\neq 0$, we have 
$Q_a(\phi (u), Y)$ is not identically zero.
\end{proof}

\begin{lemma}\label{elimination}
Let $\phi ,\psi \in \mathcalboondox{M}_{\mathbb{K},n}^n$ and suppose that $\phi$ is algebraic over $\mathbb{K}(\psi )$.
If $\phi$ admits an AAT then $\psi$ admits an AAT.
The converse is also true, provided $\phi_1,\ldots ,\phi _n$ are algebraically independent over $\mathbb{K}$.
\end{lemma}
\begin{proof}
Assume that $\phi$ admits an AAT, hence $\psi _1,\ldots ,\psi _n$ are algebraically independent over $\mathbb{K}$ because 
$\phi$ is algebraic over 
$\mathbb{K}(\psi )$.
To check that $\psi _{ u+ v}$ is algebraic over $\mathbb{K}(\psi _{(u,v)})$ it is enough to show that 
$\psi _{ u+ v}$ is algebraic over $\mathbb{K}(\phi _{ u+ v})$, $\phi _{ u+ v}$ is algebraic over 
$\mathbb{K}(\phi _{(u,v)})$ and $\phi _{(u,v)}$ is algebraic over $\mathbb{K}(\psi _{(u,v)})$.
The three conditions above are trivially satisfied because $\phi$ admits an AAT and both $\phi$ is algebraic over 
$\mathbb{K}(\psi)$ and $\psi$ is algebraic over $\mathbb{K}(\phi)$.
The converse follows by symmetry because if $\phi _1,\ldots ,\phi _n$ are algebraically independent over $\mathbb{K}$ then 
$\psi$ is algebraic over $\mathbb{K}(\phi )$.
\end{proof}

Now, we adapt to our context a result on AAT due to H.A.Schwarz, see \cite[Ch. XXI. Art. 389]{Hancock} for details.

\begin{lemma}\label{Schwarz}
Let $\epsilon >0$ and let $\phi \in \mathcalboondox{M}_{\mathbb{K},n}^n$ be convergent on $U_n(\epsilon )$ such that it admits an AAT.
Then, there exist a finite subset $\mathcal{C}\subset U_{\mathbb{K},n}(\epsilon)$, with $0\in \mathcal{C}$ and 
$\mathcal{C}=-\mathcal{C}$, and $\epsilon ' \in (0, \epsilon]$ satisfying: each element of 
$\mathbb{K}(\phi _{u+a} \suchthat a\in \mathcal{C})$ is convergent on $U_n(2\epsilon ')$,
and there exist $A_0,\ldots ,A_N\in \mathbb{K}(\phi _{(u+a,v+a)} \suchthat a\in \mathcal{C})$
convergent on $U_{2n}(2\epsilon ')$ such that $\phi _{u+v}$ is algebraic over $\mathbb{K}(A_0,\ldots , A_N)$ 
and, for each $j\in \{0,\ldots ,N\}$, 
\begin{equation}
A_j(u,v)=A_j(u+a,v-a),\text{ for all } a\in U_{\mathbb{K},n}(\epsilon ').\label{dagger}
\end{equation}
\end{lemma}
\begin{proof}
Fix $i\in \{1,\ldots ,n\}$.
Let $\mathcal{S}_0:=\{ 0\}$ and $\mathbb{K}_0:=\mathbb{K}(\phi _{(u,v)})$.
Let
\[
P_0(X)=X^{\ell _0+1}+\sum _{j=0}^{\ell _0} A_{0,j}( u, v) X^j
\]
be the minimal polynomial of $\phi _i(u+v)$ over $\mathbb{K}_0$.
If each $A_{0,j}$ satisfies equation (\ref{dagger}) for $\epsilon '=2^{-1}\epsilon$ then we are done for this $i$ letting 
$\epsilon ':=2^{-1}\epsilon$, $\mathcal{C}:=\mathcal{S}_0$ and $A_j:=A_{0,j}$, for each $0\leq j \leq \ell _0$.
Otherwise, there exists $a_1\in U_{\mathbb{K},n}(2^{-1}\epsilon )$ such that
\[
Q_0(X):= X^{\ell _0+1}+\sum _{j=0}^{\ell _0} A_{0,j}( u, v)X^j - 
X^{\ell _0+1}-\sum _{j=0}^{\ell _0} A_{0,j}(u+a_1,v-a_1)X^j 
\]
is not zero.
Since $u+v=(u+a_1)+(v-a_1)$, we deduce that $\phi _i(u+v)$ is a root of $Q_0(X)$.
Let $\mathcal{S}_1:=\mathcal{S}_0\cup \{a_1,-a_1\}$ and 
$\mathbb{K}_1:=\mathbb{K}(\phi _{u+a,v+a} \suchthat a\in \mathcal{S}_1)$.
By definition $\mathbb{K}_0\subset \mathbb{K}_1$.
Let 
\[
P_1(X)=X^{\ell _1+1}+\sum _{j=0}^{\ell _1} A_{1,j}(u,v) X^j
\]
be the minimal polynomial of $\phi _i(u+v)$ over $\mathbb{K}_1$.
We note that the elements of $\mathbb{K}_1$ are convergent on $U_{2n}(2^{-1}\epsilon )$.
If each $A_{1,j}$ satisfies equation (\ref{dagger}) for $\epsilon '=2^{-2}\epsilon$ then we are done for this $i$ letting 
$\epsilon ':=2^{-2}\epsilon$, $\mathcal{C}:=\mathcal{S}_1$ and $A_j:=A_{1,j}$, for each $0\leq j \leq \ell _1$.
Otherwise, we can repeat the process to obtain sets $\mathcal{S}_2$, $\mathcal{S}_3$ and so on, where the set 
$\mathcal{S}_k$ is obtained from the set $\mathcal{S}_{k-1}$ as
\[
\mathcal{S}_k:=\mathcal{S}_{k-1}\cup \{a+a_k \suchthat a\in \mathcal{S}_{k-1}\}\cup 
\{a-a_k \suchthat a\in \mathcal{S}_{k-1}\},
\]
for some $a_k\in U_{\mathbb{K},n}(2 ^{-k}\epsilon )$ such that $Q_{k-1}$ is not $0$. 
Similarly, we obtain $\mathbb{K}_k:=\mathbb{K}(\phi_{u+a,v+a} \suchthat a\in \mathcal{S}_k)$ whose elements are convergent on 
$U_{2n}(2^{-k}\epsilon)$.
Since in the $k$ repetition the degree of $P_k$ is smaller than that of $P_{k-1}$, this process eventually stops, say at step $s$.
Letting $\epsilon ':=2^{-s-1}\epsilon$, $\mathcal{C}:=\mathcal{S}_s$ and $A_j:=A_{s,j}$, for each $0\leq j \leq \ell _s$, 
we are done for this $i$.
The elements $A_0,\ldots ,A_{\ell _s}$ are convergent on $U_{2n}(2\epsilon ')$ because they are elements of $\mathbb{K}_s$.

For each $i$, $1\leq i \leq n$, denote by $\epsilon '_i$, $\mathcal{C}_i$ and $A_{0}^{i},\ldots ,A_{N_i}^{i}$ the elements 
$\epsilon '$, $\mathcal{C}$ and $A_{1},\ldots ,A_{\ell _s}$ previously obtained for that choice of $i$.
To complete the proof, take $\mathcal{C}:=\bigcup _i \mathcal{C}_i$, $\epsilon ' :=\min _i \{ \epsilon '_i\}$, and let 
$\{A_0,\ldots ,A_N\}$ be the union of the sets $\{A_0^{i},\ldots ,A_{N_i}^{i}\}$.
\end{proof}

We need two additional lemmas before proving the Extension Theorem.

\begin{lemma}\label{minus}
Let $\phi \in \mathcalboondox{M}_{\mathbb{K},n}^n$ admit an AAT. 
Then, $\phi (-u)$ is algebraic over $\mathbb{K}(\phi (u))$.
\end{lemma}
\begin{proof}
Take $\epsilon >0$ such that $\phi \in \mathcalboondox{M}_{\mathbb{K},n}^n$ is convergent on $U_n(\epsilon )$.
Since $\phi$ admits an AAT, we know that $\phi (u+v)$ is algebraic over $\mathbb{K}(\phi (u),\phi (v))$.
Taking into account transcendence degrees, it follows that $\phi (v)$ is algebraic over $\mathbb{K}(\phi (u+v),\phi (u))$.
For some $a\in U_{\mathbb{K},n}(\epsilon )$, we may substitute $v$ by $-u+a$, so $\phi (-u+a)$ is algebraic over $\mathbb{K}(\phi (u))$.
By Lemma \ref{algebraic}, $\phi (-u)$ is algebraic over $\mathbb{K}(\phi (-u+a))$ and hence over $\mathbb{K}(\phi (u))$. 
\end{proof}

\begin{lemma}\label{K(Psi)} 
Let $\epsilon >0$. Let $\phi \in \mathcalboondox{M}_{\mathbb{K},n}^n$ be convergent on $U_n(\epsilon )$ such that it admits an AAT.
Then there exist $\epsilon _1\in (0,\epsilon ]$ and $\Psi :=(\psi _0,\ldots ,\psi _n)\in \mathcalboondox{M}_{\mathbb{K},n}^{n+1}$ 
convergent on $U_n(\epsilon _1)$ and algebraic over 
$\mathbb{K}(\phi )$ satisfying $\psi :=(\psi _1,\ldots ,\psi _n)$ admits an AAT, 
$\psi _0$ is algebraic over $\mathbb{K}(\psi )$ and, for each $f \in \mathbb{K}(\Psi )$, $f(-u)\in \mathbb{K}(\Psi (u))$,
there exists $\delta \in (0, \epsilon _1]$ such that for each $a\in U_{\mathbb{K},n}(\delta )$, $f_{u+a}\in \mathbb{K}(\Psi)$ 
and $f_{u+a}$ is convergent on $U_n(\epsilon _1)$.
\end{lemma}
\begin{proof}
We will define a field $\mathbb{L}$ generated over $\mathbb{K}$ by certain elements of $\mathcalboondox{M}_{\mathbb{K},n}$, next we will 
prove that each $f\in \mathbb{L}$ satisfy the conclusion of the lemma and finally we find a primitive element $\Psi$ such that 
$\mathbb{L}=\mathbb{K}(\Psi)$.

Let $\epsilon ' \in (0, \epsilon]$, $\mathcal{C}\subset U_{\mathbb{K},n}(\epsilon )$ and 
$A_0,\ldots ,A_N \in \mathbb{K}(\phi _{(u+c,v+c)} \suchthat c\in \mathcal{C})$ 
be the ones provided by Lemma \ref{Schwarz} for $\phi$.
Let $U$ be an open dense subset of $U_{\mathbb{K},n}(\epsilon ')$ such that
\[
U\subset \{  a\in U_{\mathbb{K},n}(\epsilon ') \suchthat \phi (a+c)\in \mathbb{K}^n \text{ for all }  c\in \mathcal{C}\}
\]
and
\[
U\subset \{ a\in U_{\mathbb{K},n}(\epsilon ') \suchthat A_0 (u,a),\ldots ,A_N (u,a) \in \mathcalboondox{M}_{\mathbb{K},n} \}.
\]
In particular, $U\subset \{a\in U_{\mathbb{K},n}(\epsilon ') \suchthat \phi (a)\in \mathbb{K}^n\}$ because $0\in \mathcal{C}$.
Since $U$ is open there exist $b\in U$ and $\epsilon '' \in (0,\epsilon '- \| b\|]$ such that 
\[
V:=\{a\in U_{\mathbb{K},n}(\epsilon ') \suchthat \|a-b\| < \epsilon '' \}\subset U.
\]
Fix such $b$.
Then, for each $a \in U_{\mathbb{K},n}(\epsilon '' )$, each $A_j(u,a+b)$, $j=1,\ldots, N$ is an element of 
$\mathcalboondox{M}_{\mathbb{K},n}$.
We note that since each $A_j (u,v)$ is convergent on $U_{2n}(2\epsilon ')$ and by definition of $b$ and $\epsilon ''$, 
each $A_j(u,a+b)$ is convergent on $U_n(\epsilon ')$, for each $a\in U_{\mathbb{K},n}(\epsilon '' )$.
Also, since each $A_j$ satisfies the equation (\ref{dagger}) of Lemma \ref{Schwarz},
\begin{equation}
A_j(u,a+b)=A_j(u+a,b) \text{ for all }  a\in U_{\mathbb{K},n}(\epsilon '' ).\label{daggerb}
\end{equation}
For each $j\in \{0,\ldots ,N\}$, we define $B_j(u):=A_j(u,b)$.
Let
\[
\mathbb{L}_1 :=\mathbb{K}((B_j)_{u+a} \suchthat a\in U_{\mathbb{K},n}(\epsilon '' ), 0\leq j \leq N).
\]
Since, for each $a\in U_{\mathbb{K},n}(\epsilon '' )$, each $A_j(u,a+b)$ is convergent on $U_n(\epsilon ')$, 
by equation (\ref{daggerb}) all the elements of $\mathbb{L}_1$ are convergent on $U_n(\epsilon ')$ and 
in particular in $U_n(\epsilon '')$.
Let
\[
\mathbb{L}_2 :=\mathbb{K}((B_j)_{-u+a} \suchthat a\in U_{\mathbb{K},n}(\epsilon '' ), 0\leq j \leq N).
\]
Note that all the elements of $\mathbb{L}_2$ are also convergent on $U_n(\epsilon '')$.
Hence, if we define 
\[
\mathbb{L} :=\mathbb{K}((B_j)_{u+a},(B_j)_{-u+a} \suchthat a\in U_{\mathbb{K},n}(\epsilon '' ), 0\leq j \leq N),
\]
all the elements of $\mathbb{L}$ are also convergent on $U_n(\epsilon '')$.

Let us show that 
\[
\mathbb{L}\subset \mathbb{K}(\phi _{u+c},\phi _{-u+c} \suchthat  c\in \mathcal{C}) 
\]
and that each element of $\mathbb{L}$ is algebraic over $\mathbb{K}(\phi )$.

We begin proving that 
\[
\mathbb{L}_1\subset \mathbb{K}(\phi _{u+c} \suchthat c \in \mathcal{C}) 
\]
and that each element of $\mathbb{L}_1$ is algebraic over $\mathbb{K}(\phi )$.
Fix $j\in \{0,\ldots ,N\}$ and $a\in U_{\mathbb{K},n}(\epsilon '' )$.
We recall from Lemma \ref{Schwarz} that $A_j(u,v)$ is convergent on $U_{2n}(2\epsilon ')$ and 
$A(u,v)\in \mathbb{K}(\phi  _{(u+c,v+c)} \suchthat c\in \mathcal{C})$.
Hence we can evaluate $A_j(u,v)$ at $v=a+b$ to deduce that $A_j(u,a+b)\in \mathbb{K}(\phi _{u+c} \suchthat c\in \mathcal{C})$.
Thus, by equation (\ref{daggerb}), $A_j(u+a,b)\in \mathbb{K}(\phi _{u+c} \suchthat c\in \mathcal{C})$.
Hence, $\mathbb{L}_1\subset \mathbb{K}(\phi _{u+c} \suchthat c\in \mathcal{C})$ and therefore, by Lemma \ref{algebraic}, 
each element of $\mathbb{L}_1$ is algebraic over $\mathbb{K}(\phi )$.
By symmetry of $\mathcal{C}$, $\mathbb{L}_2\subset \mathbb{K}(\phi _{-u+c} \suchthat c\in \mathcal{C})$ and 
each element of $\mathbb{L}_2$ is algebraic over $\mathbb{K}(\phi (-u))$.
Therefore $\mathbb{L} \subset \mathbb{K}(\phi_{u+c,-u+c} \suchthat c\in \mathcal{C})$ and, since by Lemma \ref{minus} we have that 
$\phi(-u)$ is algebraic over $\mathbb{K}(\phi(u))$, we deduce that each element of $\mathbb{L}$ is algebraic over $\mathbb{K}(\phi(u))$, 
as required. 

Next, we show that $\phi _1(u+b),\ldots ,\phi _n(u+b)$ are algebraically independent over $\mathbb{K}$.
Let $P\in \mathbb{K}[X_1,\ldots ,X_n]$ be such that $P(\phi _{u+b})=0$.
By notation, for each $a\in U_{\mathbb{K},n}(\epsilon '')$, we have that $P(\phi _{u+b}(a))=0$ if and only if $P(\phi (a+b))=0$.
Hence,
\[
V \subset \{  a\in U_{\mathbb{K},n}(\epsilon ) \suchthat \phi (a)\in \mathbb{K} \text{ and } P(\phi (a))=0\}.
\]
Since $V$ is open in $U_{\mathbb{K},n}(\epsilon )$, $P(\phi )=0$ by the identity principle.
Since $\phi_1,\ldots ,\phi _n$ are algebraically independent over $\mathbb{K}$, $P=0$ and we are done.

Next, we show that $\mathbb{L}$ is finitely generated over $\mathbb{K}$ and its transcendence degree is $n$.
Firstly, we note that $\phi$ is algebraic over $\mathbb{K}(\phi _{u+b})$ because 
the coordinate functions of $\phi _{u+b}$ are algebraically independent over 
$\mathbb{K}$ and $\phi _{u+b}$ is algebraic over $\mathbb{K}(\phi )$ by Lemma \ref{algebraic}.
Since $\phi _{ u+ v}$ is algebraic over $\mathbb{K}(A_0,\ldots ,A_N)$, evaluating each $A_j(u,v)$ at $v=b$ we deduce that 
$\phi _{u+b}$ is algebraic over $\mathbb{K}(B_0,\ldots ,B_N)$.
Therefore, $\phi$ is algebraic over $\mathbb{K}( B_0,\ldots ,B_N)$.
On the other hand, $\mathbb{K}(B_0,\ldots ,B_N)$ is a subset of $\mathbb{K}(\phi _{u+c} \suchthat c\in \mathcal{C})$ and the 
latter field is algebraic over $\mathbb{K}(\phi )$ by Lemma \ref{algebraic}.
Hence the three fields have transcendence degree $n$ over $\mathbb{K}$.
Recall that $\mathcal{C}=-\mathcal{C}$, so 
$\mathbb{K}(\phi _{-u+c} \suchthat c\in \mathcal{C})=\mathbb{K}(\phi _{-u-c} \suchthat c \in \mathcal{C})$.
We also note that $\phi (-u)$ is algebraic over 
$\mathbb{K}(\phi (u))$, so $\mathbb{K}(\phi _{u+c}, \phi _{-u-c} \suchthat c\in \mathcal{C})$
has transcendence degree $n$ over $\mathbb{K}$.
Now, $\mathcal{C}$ is finite and 
\[
\mathbb{K}(B_0(u),\ldots ,B_N(u))\subset \mathbb{L} \subset \mathbb{K}(\phi _{u+c},\phi _{-u-c} \suchthat c\in \mathcal{C}),
\]
therefore, $\mathbb{L}$ is finitely generated over $\mathbb{K}$ and its transcendence degree is $n$.

Fix $f \in \mathbb{L}$ and let us check that $f(-u)\in \mathbb{L}$ and that there exists $\delta >0$ such that for 
every  $a\in U_{\mathbb{K},n}(\delta )$, $f_{u+a}\in \mathbb{L}$ and $f_{u+a}$ is convergent on $U_n(\epsilon '')$.

Since $f \in \mathbb{L}$, there exist $m,m'\in \mathbb{N}$, $j(1),\ldots ,j(m+m')\in \{0,\ldots ,N\}$ 
and $a_1,\ldots , a_{m+m'}\in U_{\mathbb{K},n}(\epsilon '' )$ such that $f$ is a rational function of 
\[
(B_{j(1)})_{u+a_1}, \ldots ,(B_{j(m)})_{u+a_m},(B_{j(m+1)})_{-u+a_{m+1}}, \ldots ,(B_{j(m+m')})_{-u+a_{m+m'}}.
\]
In particular, $f(-u)$ is a rational function of 
\[
(B_{j(1)})_{-u+a_1}, \ldots ,(B_{j(m)})_{-u+a_m}, (B_{j(m+1)})_{u+a_{m+1}}, \ldots ,(B_{j(m+m')})_{u+a_{m+m'}},
\]
so $f(-u)\in \mathbb{L}$.
Take $\delta >0$ such that $\delta <\epsilon '' -\max\{ \| a_1\|,\ldots ,\|  a_{m+m'}\|\}$.
Then, for all $ a\in U_{\mathbb{K},n}(\delta )$, $f _{u+a}\in \mathbb{L}$ and $f_{u+a}$ is convergent on $U_n(\epsilon '')$.

Finally, take $\psi _1,\ldots ,\psi _n\in \mathbb{L}$ algebraically independent 
over $\mathbb{K}$ and $\psi _0$ algebraic over $\mathbb{K}(\psi _1,\ldots ,\psi _n)$ such that 
$\mathbb{L}=\mathbb{K}(\psi _0,\psi _1,\ldots ,\psi _n)$.
Now, since all the elements of $\mathbb{L}$ are algebraic over $\mathbb{K}(\phi )$, $\psi:= (\psi _1,\ldots ,\psi _n)$ admits 
an AAT by Lemma \ref{elimination}.
\end{proof}

We now have all the ingredients to prove our main result.

\begin{proof}[Proof of the Extension Theorem]
Let $\phi:=(\phi _1,\ldots ,\phi _n)\in \mathcalboondox{M}_{\mathbb{K},n}^n$ admit an AAT.
Take $\epsilon >0$ such that $\phi$ is convergent on $U_n(\epsilon )$.
Applying Lemma \ref{K(Psi)} we obtain $\epsilon _1\in (0,\epsilon ]$ and 
$\Psi :=(\psi _0,\ldots ,\psi _n)\in \mathcalboondox{M}_{\mathbb{K},n}^{n+1}$ as in the lemma.
We next check that this $\Psi$ satisfies the conditions of the theorem.

$(1)$ By Lemma \ref{K(Psi)}, if $f\in \mathbb{K}(\Psi)$ then $f(-u)\in \mathbb{K}(\Psi)$, so we only have to check 
$f(u+v)\in \mathbb{K}\big(\Psi _{(u,v)}\big)$.
Fix a non-constant $f\in \mathbb{K}(\Psi )$ and $\delta \in (0, \epsilon _1]$ such that $f_{u+a}\in \mathbb{K}(\Psi )$, for each 
$a\in U_n(\delta)$, as in Lemma \ref{K(Psi)}.
Let $0<\varepsilon < \delta$ be such that $f_{u+v}$ is convergent on $U_{2n}(\varepsilon )$.
Let $U$ be an open connected subset of $U_n(\varepsilon )$ such that $\Psi(u)$ is analytic on $U$. 
In particular, $\Psi_{(u,v)}$ is analytic on $U\times U$.  
On the other hand, if for each $a\in U$ we have that $g(u,v):=f(u+v)$ is not analytic in $(a,a)$ then we would deduce that $f(u)$ 
is not analytic on an open subset of $U_n(\varepsilon)$, a contradiction.
Therefore, shrinking $U$ we can assume that $g(u,v):=f(u+v)$ is also analytic on $U\times U$.  
By Lemma \ref{K(Psi)}, we have that $g(u,a) \in \mathbb{K}(\Psi (u))$ and $g(a,v) \in \mathbb{K}(\Psi(v))$, for each $a \in U$. 
Hence, by Bochner \cite[Theorem 3]{Bochner}, $g(u,v) \in \mathbb{C}(\Psi _{(u,v)})$ on $U\times U$. 
Since $U\times U$ is an open subset of $U_{2n}(\varepsilon)$, it follows that $g(u,v) \in \mathbb{C}(\Psi _{(u,v)})$ on $U_{2n}(\varepsilon)$. 
Moreover, clearly $g(u,v) \in \mathbb{K}(\Psi _{(u,v)})$ on $U_{2n}(\varepsilon)$ since both 
$\Psi \in \mathcalboondox{M}_{\mathbb{K},n}^{n+1}$ and $f\in \mathbb{K}(\Psi)$.
This concludes the proof of $(1)$.

\smallskip

$(2)$ We may assume that $\psi _0\neq 0$.
Fix $i\in \{0,\ldots ,n\}$.
We have already shown that $\psi _i(u+v)\in \mathbb{K}(\Psi _{(u,v)})$.
Let $A(u,v):=\psi _i(u+v)$.
By Lemma \ref{K(Psi)} and taking a smaller $\epsilon >0$ if necessary, we may assume that $\Psi$ is convergent on 
$U_n(\epsilon )$ and $\mathbb{K}(\Psi _{u+a})\subset \mathbb{K}(\Psi)$, for each $a\in U_{\mathbb{K},n}(\epsilon)$.
Let us show that there exists $c\in U_{\mathbb{K},n}(\epsilon)$ such that 
\[
A(u+c,u-c)\in \mathcalboondox{M}_{\mathbb{K},n}.
\]
Take $\alpha ,\beta \in \mathcal{O}_{\mathbb{K},2n}$, $\beta \neq 0$ such that 
$A(u,v)=\frac{\alpha (u,v) }{\beta (u,v)}$.
Suppose by contradiction that $\beta (u+c,u-c)=0$ for all $c\in U_{\mathbb{K},n}(\epsilon)$.
Then 
\[
\beta \left(\frac{a+b}{2}+\frac{a-b}{2},\frac{a+b}{2}-\frac{a-b}{2}\right) =0,
\]
for all $a,b\in U_{\mathbb{K},n}(\epsilon/2)$.
So $\beta (a,b)=0$, for all $(a,b)\in U_{\mathbb{K},n}(\epsilon/2)$, that is, $\beta =0$, which is a contradiction.
Consequently,
\[
\psi _i(2u)=A(u+c,u-c)\in \mathbb{K}(\Psi _{u+c}( u),
\Psi _{u-c}(u))\subset \mathbb{K}(\Psi (u)).
\]

By induction we deduce that 
\[
\psi _0(u),\ldots ,\psi _n(u)\in \mathbb{K}(\Psi (2^{-N}u)),
\]
for each $N\in \mathbb{N}$.
Hence since $\Psi (2^{-N}u)$ is convergent on $U_n(2^N\epsilon)$, $\Psi$ is also convergent on $U_n(2^N\epsilon)$. 
Thus each $\psi_i$ is the quotient of two power series convergent in all $\mathbb{C}^n$ 
(by Poincar\'e's problem \cite[Ch. VIII, \textsection B, Corollary 10]{Gunning_Rossi}).
\end{proof}

\begin{proof}[Proof of Corollary \ref{C2}]
Let $\phi \in \mathcalboondox{M}_{\mathbb{K},n}^n$ admit an AAT.
By Theorem \ref{T2}, there exists $\psi \in \mathcalboondox{M}_{\mathbb{K},n}^n$ admitting an AAT whose coordinate functions are the 
quotient of two convergent whose complex domain of convergence is $\mathbb{C}^n$ and such that $\psi$ is algebraic over 
$\mathbb{K}(\phi)$.
Since the coordinate functions of $\psi$ are algebraically independent, $\phi$ is algebraic over $\mathbb{K}(\psi)$.
\end{proof}

\section*{Acknowledgements}

The second author thanks E. Pantelis for the support to attend ``Summer School in Tame Geometry'', 
Konstanz, July 18-23, 2016, where the results of this paper were presented.
The authors also would like to thank José F. Fernando for helpful suggestions on an earlier version of this paper and Mark Villarino 
for his comments.

\bibliographystyle{abbrv}
\bibliography{research}

\end{document}